\documentclass[12pt]{article}
\usepackage{amssymb}
\usepackage{amsmath,amsthm}
\usepackage{amsfonts,amssymb,amscd,amsmath,enumerate,verbatim,calc}
\usepackage{graphicx}
\usepackage{enumerate}
\usepackage{multirow}
\usepackage{mathrsfs}
\usepackage{caption, subcaption}

\newtheorem{remark}{Remark}

\newtheorem{lemma}[remark]{Lemma}
\newtheorem{theorem}[remark]{Theorem}
\newtheorem{proposition}[remark]{Proposition}
\newtheorem{corollary}[remark]{Corollary}

\title{\bf Weak total resolving sets in graphs}

\author{I. Javaid$^{1}\thanks{Corresponding author email: ijavaidbzu@gmail.com}$, M. Salman$^1$, M. Murtaza$^1$,
F. Iftikhar$^1$ and M. Imran$^2$}

\date{$^1$ Center for Advanced Studies in Pure and Applied Mathematics,
Bahauddin Zakariya University Multan, 60800, Pakistan\\ Email:
\{ijavaidbzu, solo33, mahru830, farheen689\}@gmail.com\\ $^2$ School
of Natural Sciences, National University of Science and Technology,
Islamabad, Pakistan, Email: imrandhab@gmail.com}

\begin{document}

\maketitle

\begin{abstract}
A set $W$ of vertices of $G$ is said to be a weak total resolving
set for $G$ if $W$ is a resolving set for $G$ as well as for each
$w\in W$, there is at least one element in $W-\{w\}$ that resolves
$w$ and $v$ for every $v\in V(G)- W$. Weak total metric dimension of
$G$ is the smallest order of a weak total resolving set for $G$.
This paper includes the investigation of weak total metric dimension
of trees. Also, weak total resolving number of a graph as well as
randomly weak total $k$-dimensional graphs are defined and studied
in this paper. Moreover, some characterizations and realizations
regarding weak total resolving number and weak total metric
dimension are given.
\end{abstract}



\section{Introduction}
Unless otherwise specified, all the graphs $G$ considered in this
paper are simple, non-trivial and connected with vertex set $V(G)$
and edge set $E(G)$. Two adjacent vertices $u, v$ in $G$ will be
denoted by $u\sim v$, and non-adjacent vertices $u$, $v$ will be
denoted by $u\not\sim v$. The subgraph induced by a set $S$ of
vertices of $G$ is denoted by $\langle S\rangle$. Two isomorphic
graphs $G$ and $H$ are denoted by $G\cong H$. The {\it neighborhood}
of a vertex $v$ of $G$ is the set $N(v) = \{u\in V(G)\ |\ u\sim
v\}$. The number of elements in $N(v)$ is the {\it degree} of $v$,
denoted by $d(v)$. The maximum degree of $G$ is denoted by
$\Delta(G)$. If two distinct vertices $u$ and $v$ of $G$ have the
property that $N(u)- \{v\} = N(v)- \{u\}$, then $u$ and $v$ are
called {\it twin vertices} (or simply twins) in $G$. If for a vertex
$u$ of $G$, there exists a vertex $v\neq u$ in $G$ such that $u, v$
are twins in $G$, then $u$ is said to be a {\it twin} in $G$. If
$\langle N(v)\rangle \cong K_{d(v)}$ for $v\in V(G)$, then $v$ is
called a {\it complete vertex} in $G$. The number $d(u,v)$ denotes
the {\it distance} between two vertices $u$ and $v$ of $G$, which is
the number of edges in a shortest $u - v$ path in $G$. The maximum
distance between two vertices in $G$ is called the {\it diameter} of
$G$, denoted by $diam(G)$. Two vertices $u, v$ of $G$ are said to be
{\it antipodal} if $d(u,v) = diam(G)$ otherwise, $u$ and $v$ are
called {\it non-antipodal}. A vertex of degree one is called a {\it
leaf} in $G$. A vertex of degree at least three in $G$ is called a
{\it major vertex}. An end vertex $u$ is a {\it terminal vertex} of
a major vertex $v$ such that $d(u,v) < d(u,w)$ for each other major
vertex $w$. The number of terminal vertices for a major vertex $v$
is its terminal degree $td(v)$. If $td(v)>0$ for a major vertex $v$,
then $v$ is called an {\it exterior major vertex}. The sum $\sum
td(v)$ taken over all the major vertices $v$ of $G$ is denoted by
$\sigma(G)$, and $ex(G)$ denotes the number of exterior major
vertices of $G$. The symbol $X \bigtriangledown Y$ denotes the {\it
symmetric difference} of two sets $X$ and $Y$.

 A vertex $x$ of $G$ {\it resolves} (or
{\it distinguishes}) two distinct vertices $y, z$ of $G$ if $d(y,x)
\neq d(z,x)$. The $k$-tuple $c_U(v) =
(d(v,u_1),d(v,u_2),\ldots,d(v,u_k))$ is the {\it code of} $v$ with
respect to a set $U = \{u_1,u_2,\ldots,u_k\}\subseteq V(G)$. A
subset $W$ of $V(G)$ is called a {\it resolving set} for a graph $G$
if for every two distinct vertices $u$ and $v$ of $G$, there is an
element $w$ in $W$ that resolves $u$ and $v$. Equivalently, the set
$W$ is a resolving set if for every two vertices $u$ and $v$ of $G$,
we have $c_{W}(u)\neq c_{W}(v)$. The minimum cardinality of a
resolving set for $G$ is called the {\it metric dimension} of $G$,
denoted by $\dim(G)$. A resolving set of cardinality $\dim(G)$ is
called a {\it metric basis} of $G$. This concept was firstly studied
by Slater in 1975 in \cite{slater1975} with the names {\it locating
set \emph{and} location number} rather than resolving set and metric
dimension. The terminologies, we used above for this concept and
will be used throughout this paper, was proposed by Harary and
Melter when they independently studied this concept in 1976
\cite{harary1976}. This concept has wide range of applications not
just in graph theory but to many other fields. For instance,
 Khuller {\em et al.} studied its application in robot navigation \cite{khuller1994}; Melter and
Tomescu studied its application in pattern recognition and image
processing \cite{melter1984}, to name a few. First time, in 1979,
Gary and Johnson noted that to determine the metric dimension of a
graph is an NP-hard problem \cite{gary1979}, however, its explicit
construction was given by Khuller {\em et al.} in 1996
\cite{khuller1996}.

Elements of a metric basis were referred to as sensors in many
applications. If one of the sensors, say $s$, placed at a location
stops working due to any impenetrable difficulty, then we will not
receive enough information regarding the detection of those two
locations where no sensor is placed and they are only be detected by
the sensor $s$. This kind of problem was solved by defining
fault-tolerant resolving set, which was defined by Hernando {\em et
al.} in 2008 \cite{hernando2008} in the following way: a resolving
set $W$ for a graph $G$ is fault-tolerant if $W\setminus\{v\}$ is
also resolving set for each $v$ in $W$. But, a problem still to be
addressed is: {\it let the two locations, say $L_1$, $L_2$, on which
$L_1$ is where the sensor $s$ placed and $L_2$ is where no sensor
placed, and let they are only be detected by the sensor $s$. Then
which of the other placed sensors provides the complete information
regarding the detection of the locations $L_1$ and $L_2$ if the
sensor $s$ stops working?} The answer of this kind of problems leads
to introduce the concept of total resolving set. This concept was
introduced by Javaid {\em et al.} in 2012 \cite{javaid2012} in the
following way: a resolving set for a graph $G$ is called a total
resolving set, written as TR-set, if for every pair of distinct
vertices $u,v$ in $G$, there is a vertex $w$ in $W$ such that
$d(u,w)\ne d(v,w)$ for $u,v\ne w$ (it was named strong total
resolving set in \cite{javaid2012}, but in analogy with total
dominations in graphs in \cite{Cockayne}, we use the term total
resolving sets). But complete graphs or graphs with twin vertices do
not have any total resolving sets. On relaxing a condition in total
resolving set, a new parameter, named as weak total resolving set,
defined by Javaid {\em et al.} in \cite{javaid2012} as: a resolving
set $W$ for a graph $G$ is called a {\it weak total resolving set},
simply written as WTR-set, if for every pair of distinct vertices
$u,v$ of $G$ with $u\in V(G)- W$ and $v\in W$, there is a vertex
$w\in W- \{v\}$  such that $d(u,w)\neq d(v,w)$. Here, in this paper,
we extend the study of this concept and define and study some new
parameters in the context of weak total resolvability. One thing
important to note here is that one might think that the two
concepts: the concept of fault-tolerant resolving set and the
concept of weak total resolving set, are equivalent. But, a
comparison between these two concepts, given in \cite{javaid2012},
shows that they are not equivalent.

 The cardinality of a minimum WTR-set is called the {\it weak total metric dimension}
(WTMD) of $G$, denoted by $\dim_{wt}(G)$. A WTR-set of cardinality
$\dim_{wt}(G)$ is called a {\it weak total metric basis} (WTMB) of
$G$. The {\it weak total resolving number} of a graph $G$, denoted
by $res_{wt}(G)$, is the minimum positive integer $r$ such that
every $r$-set of vertices of $G$ is a WTR-set for $G$. Chartrand and
Zhang in \cite{chartrandzhang2000} considered graphs $G$ with
$\dim(G) = res(G)$. They called these graphs randomly
$k$-dimensional graphs, where $k = \dim(G)$. We say that a graph $G$
is \textit{randomly weak total $k$-dimensional} if $\dim_{wt}(G) =
res_{wt}(G)=k$. The aim of this paper is to study the three
parameters $\dim_{wt}(G)$, $res_{wt}(G)$ and randomly weak total
$k$-dimensional graph. We investigate $\dim_{wt}(G)$ and
$res_{wt}(G)$ when $G$ is a tree. Also, we reveal some properties of
graphs $G$ having $\dim_{wt}(G) = 2$ and characterize all the graphs
$G$ with $\dim_{wt}(G) = |G|$. Moreover, we classify the graphs $G$
with $res_{wt}(G) = 3$ and $res_{wt}(G) = |G|$.

\section{Weak total metric dimension}
The following useful result for finding the weak total metric
dimension of graphs was proposed in \cite{javaid2012}.

\begin{lemma}{\em \cite{javaid2012}}\label{lem1}
A resolving set $W$ for a graph $G$ is a {\em WTR}-set if and only
if the code, with respect to $W$, of each $x\in V(G)-W$ differ by at
least two coordinates from the code, with respect to $W$, of each
$w\in W$.
\end{lemma}

Let $u,v$ be two twins in a graph $G$ and let $W$ be a WTR-set for
$G$ such that $u\in W$ and $v\not \in W$. Since $d(u,w) = d(v,w)$
for all $w\in V(G)-\{u,v\}$, so the codes of $u$ and $v$, with
respect to $W$, differ by one coordinate only. Thus, Lemma
\ref{lem1} concludes the following remark:

\begin{remark}\label{rem1}
Every WTR-set for a graph $G$ contains all the twins of $G$.
\end{remark}

Since every WTR-set for a graph $G$ is a resolving set as well, so
the following is trivial relationship between $\dim(G)$ and
$\dim_{wt}(G)$

\begin{equation}\label{eq1}
\dim(G)\leq \dim_{wt}(G).
\end{equation}

To see the equality holds in (\ref{eq1}), construct a graph as
follows: Take two copies of $K_3$ with vertex sets $\{u,v,w\}$ and
$\{x,y,z\}$. Take a path $P_{a\geq 2}$ with its two leaves called
$l_1, l_2$. By an edge attach one leaf $l$ with $u$ and one leaf
$l'$ with $x$. Identify the leaf $l_1$ with the vertex $v$ and the
leaf $l_2$ with the vertex $y$. Call the resultant graph $G$. Since
$G$ is not a path and no single vertex forms a WTR-set for $G$, so
$\dim(G)\geq 2$ and $\dim_{wt}(G)\geq 2$. Further, the set
$\{l,l'\}$ is a resolving set as well as WTR-set for $G$ and, as a
consequence, $\dim(G) = \dim_{wt}(G) = 2$. To see the inequality
holds in (\ref{eq1}), let the graph $G = K_r + (K_1 \cup K_s)$ of
order $n = r+s+1$ for $r,s\geq 2$ ($G+H$ is the join (sum) of the
graph $G$ and $H$). It was investigated in \cite{chartrand2000} that
$\dim(G) = n-2$. We claim that $\dim_{wt}(G) = n-1$ or $n$. But,
$\dim_{wt}(G)\neq n$ because the set $V(G)-V(K_1)$ is a WTR-set for
$G$. Since each element of the sets $V(K_r)$ and $V(K_s)$ is twin in
$G$, Remark \ref{rem1} concludes that $\dim_{wt}(G)\geq r+s = n-1$,
and, as a consequence, $\dim(G)< \dim_{wt}(G) = n-1$.

Due to the relationship (\ref{eq1}), we have the following
assertions:

\begin{proposition}\label{prop1}
If $m$ is the weak total metric dimension of a graph $G$ and $D =
diam(G)$, then the following assertions hold:\\
(1)\ The maximum order of $G$ is $D^m+m$.\\
(2)\ the maximum degree of $G$ is at most $3^m-1$.\\
(3)\ $G$ is at most $2^m$-colorable.
\end{proposition}

\begin{proof}
Let $k$ denotes the metric dimension of $G$. Then $D,k$ and $m$ all
are positive integers.

Khuller {\em et al.} proved in \cite{khuller1996} that $|G|\leq
D^k+k$. Thus, the inequality (\ref{eq1}) implies part (1) of the
theorem.

Chartand {\em et al.} proved in \cite{poisson2000} that $k\geq
\log_3(\Delta(G)+1)$ ($i.e$, $\Delta(G)\leq 3^k -1$). Therefore, the
inequality (\ref{eq1}) follows the part (2) of the theorem.

Let $\chi(G)$ denotes the minimum number of colors needed to color
the graph $G$ properly ($i.e$, chromatic number of $G$). It was
shown by Chappell {\em et al.} in \cite{chappell2008} that
$\chi(G)\leq 2^k$ and, as a result, $G$ is at most $2^m$-colorable,
by the inequality (\ref{eq1}).
\end{proof}

\section{Graphs $G$ with $\dim_{wt}(G) = 2$ and $\dim_{wt}(G) = n$}
The following result proved by Khuller {\em et al.} in
\cite{khuller1996} is useful.

\begin{proposition}{\em \cite{khuller1996}}\label{prop3}
Let $G$ be a graph and $u,v, w$ be three vertices of $G$ such that
$u\sim v$. If $d(v,w) = d$, then $d(u,w)\in \{d-1,d,d+1\}$.
\end{proposition}

Now, we explore some properties of graphs with $\dim_{wt}(G)=2$ in
the next few results.

\begin{proposition}\label{prop4}
Let $G$ be a graph of order at least three. If $\{u,v\}$ be a WTMB
of $G$, then $u\not\sim v$.
\end{proposition}

\begin{proof}
If $u\sim v$, then the codes (with respect to the set $\{u,v\}$) of
$u$ and the other neighbor of $v$ differ by one coordinate only, and
hence $\{u,v\}$ is not a WTR-set for $G$, by Lemma \ref{lem1}.
\end{proof}

\begin{corollary}\label{cor}
If $\dim_{wt}(G) = 2$ and $\{u,v\}$ be a WTMB of $G$, then $u\sim v$
if and only if $G$ is isomorphic to $K_2$.
\end{corollary}

\begin{proposition}\label{prop5}
Let $G$ be a graph of order at least four. If $\dim_{wt}(G) = 2$ and
$\{u,v\}$ be a WTMB of $G$, then $u$ and $v$ are not twins in $G$.
In fact, $G$ has no twin.
\end{proposition}

\begin{proof}
Suppose that $u$ and $v$ are twins in $G$. If $u\sim w$ and $v\sim
w$ for all $w\in V(G)-\{u,v\}$, then $\{u,v\}$ is not a resolving
set for $G$. If $\{u,v\}$ is a resolving set for $G$, then the code
(with respect to the set $\{u,v\}$) of a neighbor of a common
neighbor of $u$ and $v$ differ by one coordinate only from the codes
of $u$ and $v$, a contradiction by Lemma \ref{lem1}.
\end{proof}

\begin{corollary}\label{cor1}
If $\dim_{wt}(G) = 2$, then $G$ has twins if and only if $G$ is
isomorphic to $P_2$ or $P_3$.
\end{corollary}

\begin{theorem}\label{thm2}
If $\dim_{wt}(G) = 2$ and $\{u,v\}$ be a WTMB of $G$, then degree of
$u$ and $v$ is at most two.
\end{theorem}

\begin{proof}
Let $d(u,v) = p$, then $p\geq 2$ by Proposition \ref{prop4}.
Proposition \ref{prop3} yields that the second (first) coordinate of
the code, with respect to the set $\{u,v\}$, of a neighbor of $u$
(or $v$) is one of $\{p-1,p,p+1\}$. But, $(1,p)$ (or $(p,1)$) cannot
be the code of any neighbor of $u$ (or $v$), by Lemma \ref{lem1}.
Since $\{u,v\}$ is a resolving set so both $u$ and $v$ can have at
most two neighbors.
\end{proof}

\begin{theorem}\label{thm3}
If $\dim_{wt}(G) = 2$ and $\{u,v\}$ be a WTMB of $G$, then the
followings are true:\\
(1)\ The geodesic $P$ between $u$ and $v$ is unique.\\
(2)\ Each neighbor of $u$ and $v$ has degree at most three.\\
(3)\ Every vertex on $P$, other than $u, v$ and their neighbors, has
degree at most five.\\
(4)\ The maximum degree of $G$ is at most eight.\\
(5)\ For any $w\in \{u,v\}$ and for each $z\in N(w)$, $r\not\sim w$
for all $r\in N(z)$.
\end{theorem}

\begin{proof}
    (1)\ Suppose there are two geodesics $P$ and $Q$ between $u$ and
$v$. Then, clearly, there are two vertices $x, y$ on $P,Q$,
respectively, such that $u$ is equidistant from $x$ and $y$ and, as
a result, $c_{\{u,v\}}(x) = c_{\{u,v\}}(y)$, a contradiction to the
fact that $\{u,v\}$ is a resolving set for $G$.

    (2)\ Let $x$ be a neighbor of $u$ and $d(x,v) = t$, then $d(u,v)\in \{t-1,t,t+1\}$ by
Proposition \ref{prop3}. But, $d(u,v)\neq t$, for otherwise, $u$ and
$v$ do not form WTR-set for $G$, by Lemma \ref{lem1}. Thus $d(u,v) =
t-1$ or $t+1$. Assume, without any distress, that $d(u,v) = t-1$.
Then out of three possibilities for the second coordinate of the
codes of other neighbors of $x$, only two possibilities are left. It
follows that $x$ can have at most two more neighbors (as its one
neighbor is $u$) since $\{u,v\}$ is a WTR-set for $G$.

    (3)\ Let $y$ be a vertex on the unique geodesic $P$ between $u$ and $v$ such that
$y\not \in \{u,v\}\cup N(u)\cup N(v)$. Let $(a,b)$ be the code of
$y$, then $d(u,v) = a+b$. The possible codes for the neighbors of
$y$ are: $(a-1,b-1), (a-1,b), (a-1,b+1), (a,b-1), (a,b), (a,b+1),
(a+1,b-1), (a+1,b), (a+1,b+1)$. Out of these nine pairs, the pairs
$(a-1,b-1),(a-1,b),(a,b), (a,b-1)$ cannot be the code of any
neighbor of $y$. Otherwise, either the geodesic between $u$ and $v$
is not unique or $\{u,v\}$ is not a resolving set for $G$. It
follows that $d(y)\leq 5$.

    (4)\ Since $\dim_{wt}(G) = 2$, so the result follows by Proposition
\ref{prop1}.

    (5)\ By part (2), let $N(z) = \{u,i,j\}$ for any $z\in N(u)$. Since
the graphs considered in this paper are simple, we claim that
neither $i\sim u$ nor $j\sim u$. Let $d(z, v) = q$, then the
possible second coordinate of the code of $u$ is one of
$\{q-1,q+1\}$. Without loss of generality, assume that $q-1$ is the
second coordinate of the code of $u$, $q$ is the second coordinate
of the code of $i$ and $q+1$ is the second coordinate of the code of
$j$. Now, if $i\sim u$, then $c_{\{u,v\}}(z) = c_{\{u,v\}}(i)$, and
if $j\sim u$, then $d(j,v)$ would be $q$ rather than $q+1$ and hence
$c_{\{u,v\}}(z) = c_{\{u,v\}}(j)$.
\end{proof}

\begin{theorem}\label{thm4}
If $\dim_{wt}(G) = 2$, then $G$ cannot have a complete vertex of
degree more than three.
\end{theorem}

\begin{proof}
Assume that $G$ has a complete vertex of degree more than three. The
$G$ has a complete graph of order $\geq 5$ as a subgraph. But, since
a WTR-set for $G$ of cardinality two is also a resolving set for $G$
and Khuller {\em et al.} in \cite{khuller1996} proved that any graph
having a resolving set of cardinality two cannot have $K_5$ as a
subgraph. It completes the proof.
\end{proof}

Since a graph having a resolving set of order two cannot have $K_5$
as a subgraph \cite{khuller1996}, so we have the following
straightforward consequences:

\begin{corollary}\label{cor2}
If $\dim_{wt}(G)=2$, then\\
(1)\ $G$ is at most 4-colorable,\\
(2)\ the clique number of $G$ is at most 4.
\end{corollary}

In the following result, we classify the graphs having weak total
metric dimension equals to the order of the graphs.

\begin{theorem}\label{thm5}
A graph $G$ of order $n\geq 2$ has $\dim_{wt}(G) = n$ if and only if
each vertex of $G$ is twin.
\end{theorem}

\begin{proof}
For $n = 2$, the result is trivial so we consider $n\geq 3$. Assume
that each vertex of $G$ is twin. Let $u$ be any arbitrary vertex of
$G$. Since $u$ is twin, there is a vertex $v\neq u$ in $G$ such that
$d(u,w) = d(v,w)$ for all $w\in V(G)-\{u,v\}$. It follows that
$V(G)-\{u\}$ is not a WTR-set for $G$. Hence, $\dim_{wt}(G) = n$.

Now, assume that $\dim_{wt}(G)=n$. If we suppose that a vertex $x$
is not twin in $G$, then the code of every element of the set
$V(G)-\{x\}$ with respect to this set differ by at least two
coordinates from the code of $x$ with respect to $V(G)-\{x\}$. It
follows that the set $V(G)-\{x\}$ is a WTR-set for $G$, a
contradiction.
\end{proof}

\section{Weak total resolving number and randomly weak total $k$-dimensional graphs}

Now, we proceed to characterize the graphs where the bounds $3\leq
res_{wt}(G)\leq n$ are achieved. First of all, we obtain the lower
bound.

\begin{proposition}\label{LowerBound}
For any graph $G$ of order $n\geq 3$, $res_{wt}(G)\geq 3$.
\end{proposition}

\begin{proof}
Let $x,y\in V(G)$ be two vertices such that $x\sim y$ and let $z\in
V(G)-\{x,y\}$ be a vertex such that $z\sim y$. Since
$d(x,y)=1=d(y,z)$, the set $\{x,y\}$ is not a WTR-set for $G$ and,
as a consequence, $res_{wt}(G)\geq 3$.
\end{proof}

\begin{theorem}\label{TwinsRES(G)=n}
Let $G$ be a graph of order $n\ge 3$. Then  $res_{wt}(G)=n$ if and
only if $G$ contains a twin.
\end{theorem}

\begin{proof}
If $G\cong K_n$ or $G\cong P_3$, then $res_{wt}(G)=n$ and we are
done. So we assume that $n\ge 4$ and  $G\not \cong K_n$.

Let $u$ be a twin in $G$, then there exists a vertex $v\neq u$ such
that $u$ and $v$ are twin vertices of $G$. For any $k\in
\{3,\ldots,n-1\}$, there exists a $k$-set of vertices $W$ such that
$u\in W$ and $v\not \in W$, and then for any $w\in W-\{u\}$,
$d(w,u)=d(w,v)$. Hence, $res_{wt}(G)=n$.

Now, we assume that $G$ has no twin and let $v\in V(G)$. In this
case for any $u\in V(G)-\{v\} $, there exists $w\in
N(u)\bigtriangledown N(v)-\{u,v\}$. Since  $d(w,u)\ne d(w,v)$, we
have that  $V(G)-\{v\}$ is a WTR-set and, as a consequence
$res_{wt}(G)\leq n-1$.
\end{proof}

From Theorem \ref{thm5}, we deduce the following consequence:

\begin{corollary}\label{cor3}
A graph $G$ of order $n$ is a randomly weak total $n$-dimensional
graph if and only if each vertex of $G$ is twin.
\end{corollary}

Also, from Theorem \ref{TwinsRES(G)=n} and Corollary \ref{cor3}, we
conclude that a graph $G$ of order $n$ is a randomly weak total
$(n-1)$-dimensional graph if and only if $\dim_{wt}(G) = n-1$ and
$G$ has no twin.

In order to give a characterization of the graphs with
$res_{wt}(G)=3$, we present the following lemma:

\begin{lemma}\label{LemmaMaximumCommunNeigh}
If $res_{wt}(G) = k$, then every two vertices of $G$ have at most $k
- 2$ common neighbors.
\end{lemma}

\begin{proof}
We proceed by contradiction. Let $res_{wt}(G) = k$ and suppose that
there exist $u,v\in V(G)$ such that $|N(u)\cap N(v)|\ge k-1$. Let
$W$ be a set composed by $u$ and $k-1$ elements of $N(u)\cap N(v).$
Then for any $w\in W-\{u\}$, we have that $d(w,u)=d(w,v)$, which is
a contradiction because $|W|=k$ and $res_{wt}(G) = k$.
\end{proof}

\begin{theorem}\label{The case res_w=3}
Let $G$ be a graph. Then $res_{wt}(G) = 3$ if and only if $G$ is a
cycle graph of odd order or $G$ is a path of order greater than or
equal to three.
\end{theorem}

\begin{proof}
If $G$ is an odd cycle or a path of order $n\ge 3$, then any pair of
vertices is a resolving set. Hence, any set composed by three
vertices of $G$ is a WTR-set  and, by Proposition \ref{LowerBound},
we conclude that $res_{wt}(G)=3$.

On the other hand, if $G\cong C_n$, where $n$ is even, then for any
pair of antipodal vertices $x,y$ and the neighbors of $x$, say $a$
and $b$, we have that $d(a,x)=d(b,x)=1$ and
$d(a,y)=d(b,y)=\frac{n}{2}-1$, and then $\{x,y,a\}$ is not a WTR-set
for $G$. Thus,  $res_{wt}(G)\ne 3$.

It remains to show that if $res_{wt}(G)=3$, then $\Delta(G)\le 2$.
Suppose that there exist four different vertices $a,b,c,d\in V(G)$
such that $b,c,d\in N(a)$. In such a case, we differentiate the
following cases for the subgraph induced by the set $\{b,c,d\}$:
 \\
 \\
\noindent Case 1: $\langle \{b,c,d\} \rangle \cong N_3$ (an empty
graph). Since $d(b,c)=d(b,d)=2$ and $d(a,c)=d(a,d)=1$, we conclude
that $\{a,b,c\}$ is not a WTR-set for $G$.
 \\
 \\
\noindent Case 2: $\langle \{b,c,d\} \rangle \cong K_2\cup K_1$. We
assume, without loss of generality, that $c\sim d$. In this case,
$d(b,c)=d(b,d)=2$ and $d(a,c)=d(a,d)=1$ and so we conclude that
$\{a,b,c\}$ is not a WTR-set for $G$.
 \\
 \\
\noindent Case 3: $\langle \{b,c,d\} \rangle \cong P_3$. Now, we
assume, without loss of generality, that $c\not\sim d$. In this
case, $|N(a)\cap N(b)|\ge 2$ and, by Lemma
\ref{LemmaMaximumCommunNeigh}, we conclude that $res_{wt}(G)\ne 3$.
 \\
 \\
\noindent Case 4: $\langle \{b,c,d\} \rangle \cong K_3$. As above,
$|N(a)\cap N(b)|\ge 2$ and, by Lemma \ref{LemmaMaximumCommunNeigh},
we conclude that $res_{wt}(G)\ne 3$.

According to the above cases, we deduce that if $res_{wt}(G)=3$,
then $\Delta(G)\le 2$. Therefore, the result follows.
\end{proof}

It is easy to check that the two extremes of any path graph form a
WTR-set. Also, it was shown in \cite{javaid2012} that for any cycle
graph $\dim_{wt}(C_n)=3.$  Therefore, Theorem \ref{The case res_w=3}
leads to the following corollary:

\begin{corollary}\label{cor4}
A graph is randomly weak total $3$-dimensional if and only if it is
a cycle graph of odd order.
\end{corollary}

The maximum degree of a graph according to its weak total resolving
number is investigated in the next result.

\begin{theorem}\label{thm8}
If $res_{wt}(G)=k$, then maximum degree of $G$ is at most
$2^{k-1}+k-1$.
\end{theorem}

\begin{proof}
Let $u$ be a vertex of $G$ with $d(u) = \Delta(G)$ and let a set $U
= \{u, u_1,u_2,\ldots,u_{k-1}\}$ of order $k$, where
$u_1,u_2,\ldots,u_{k-1}\in N(u)$. Since $res_{wt}(G)=k$, so $U$ is a
WTR-set for $G$. Indeed for $x\in N(u)-U$, $d(x,u)=1$ and $d(x,u_i)
= 1$ or $2$ $(1\leq i\leq k-1)$. It follows that the maximum number
of distinct codes with respect to $U$ for the elements of $N(u)-U$
is $2^{k-1}$. Thus $|N(u)-U|\leq 2^{k-1}$ and, as a consequence,
$\Delta(G) = d(u)\leq 2^{k-1}+k-1$.
\end{proof}

The next result gives the realization of weak total metric dimension
and weak total resolving number of some connected graphs.

\begin{theorem}\label{thm9}
For every two natural numbers $a,b$ with $3\leq a\leq b$, there
exists a graph $G$ such that $\dim_{wt}(G) = a$ and $res_{wt}(G) =
b$.
\end{theorem}

\begin{proof}
For $a = b$. Let $G$ be a graph of order $b$ in which each vertex is
twin. Then $\dim_{wt}(G) = a = b = res_{wt}(G)$.

For $a = 3$ and $b\geq 4$. Consider the complete graph $K_3$ with
vertex set $\{s,t,u\}$ and the path $P_{b-a+1}$ with one leaf called
$l$. Make the graph $G$ of order $b$ by identifying the leaf $l$ of
$P_{b-a+1}$ with the vertex $u$ of $K_3$. Note that, $G$ has two
twins $s, t$ and $s\sim t$. Since only two adjacent vertices do not
form WTR-set for $G$, by Propositions \ref{prop4}, so
$\dim_{wt}(G)\geq 3$. Clearly, the vertex set of $K_3$ is a WTR-set
for $G$ and, as a result, $\dim_{wt}(G) = a$. Moreover $res_{wt}(G)
= b$, by Theorem \ref{TwinsRES(G)=n}.

For $4\leq a = b-1$. Take a vertex $v$ and attach $a$ leaves
$l_1,l_2,\ldots,l_a$ by an edge with $v$. The resultant graph $G$ is
a star graph and has order $a+1$. Since all the leaves are twins and
the collection of all these twins forms a WTR-set for $G$, so Remark
\ref{rem1} and Theorem \ref{TwinsRES(G)=n} yield that $\dim_{wt}(G)
= a$ and $res_{wt}(G) = a+1$.

For $4\leq a\leq b-2$. Consider the graph $K_4-e$ (obtained by
deleting one edge $e$ from $K_4$) with vertex set $\{w,x,y,z\}$ and
$e = y\sim z$. Attach $a-2$ leaves $l_1,l_2,\ldots,l_{a-2}$ by an
edge with the vertex $w$ of $K_4-e$. Also, identify a leaf $l$ of
the path $P_{b-a-1}$ with the vertex $x$ of $K_4-e$ ($P_1\cong K_l$
with unique vertex $l$). Call the resultant graph $G$. Clearly,
order of $G$ is $b$. Since $y,z$ and all the leaves
$l_1,l_2,\ldots,l_{a-2}$ are twins in $G$ and the set
$\{y,z,l_1,l_2,\ldots,l_{a-2}\}$ is a WTR-set for $G$, it follows
that $\dim_{wt}(G) = a$ and $res_{wt}(G) = b$, by Remark \ref{rem1}
and Theorem \ref{TwinsRES(G)=n}.
\end{proof}
\section{Weak total metric dimension and weak total resolving number
of a non-path tree}
Let $G$ be a non-path tree (we call a tree which is not a path, a
non-path tree) and $\{v_k: 1\le k\le ex(G) \}$ be the set of
exterior major vertices of $G$. If two or more paths start from an
exterior major vertex $v_k$ and end at different terminal vertices
of $v_k$, then they are called the \textit{branches} rooted at
$v_k$. Let $P_{i}^k$ : $u_{i,1}^k=v_k,u_{i,2}^k,...,u_{i,l_i^k}^k$
$(1\le i \le t_k)$ be the $t_k\ge 2$ branches of $v_k$ where $l_i^k$
is the number of vertices in branch $P_{i}^k$ and indices $i,j,k$ in
$u_{i,j}^k$ represent that vertex is at $j^{th}$ position in
$i^{th}$ branch $P_i^k$ of $k^{th}$ exterior major vertex $v_k$. In
our later discussion, absence of index $k$ in $P_i^k$, $u_{i,j}^k$
and $t_k$ means that only one exterior major vertex $v$ is under
consideration. For our convenience, we label branches $P_i^k$ in
ascending order of $i$ so that for any two branches $P_r^k$,
$P_s^k$, $r<s$ if $l_r^k\le l_s^k$. If $l_i^k=l^k$ $(1\le i \le
t_k)$ and $t_k\ge 2$, then branches $P_{i}^k$ are called
\textit{similar branches} of $v_k$ each of length $l^k-1$. If
$l_i^k=2$ $(1\le i \le t_k)$ and $t_k\ge 2$ for an exterior major
vertex $v_k$, then similar branches of $v_k$ are called \textit{twin
leaves}. Let $P_i^k$ and $P_s^k$ be two similar branches of an
exterior major vertex $v_k$, then vertex $u_{i,j}^k$ in $P_i^k$ is
at same position as $u_{s,j}^k$ in $P_s^k$. We name another branch
which plays an important role in finding WTMB of a non-path tree. A
branch $P_1^k$ of an exterior major vertex $v_k$ is called \emph{the
unique branch of shortest length} if $l_1^k<l_i^k$ for all $(2\le
i\le t_k)$. In fact, if an exterior major vertex has twin leaves,
then it does not have the unique branch of shortest length.

\begin{remark}\label{ThrmTerminalPathResolvingSet}
Let $G$ be a non-path tree and $v$ be an exterior major vertex of
$G$ with $t\ge 2$ branches $P_i$ $(1\le i \le t)$. If a set
$W\subseteq V(G)$ is a resolving set for $G$, then $W$ contains at
least one vertex other than $v$ from at least $t-1$ branches of $v$.
Let $W\cap \{V(P_{t-1})\cup V(P_{t})\}=\emptyset$, then the code of
$u_{t-1,1}$ and $u_{t,1}$ with respect to $W$ become same as these
two vertices are at the same distance from vertices of $W$, a
contradiction that $W$ is a resolving set.
\end{remark}

\begin{proposition}\label{prop2}
Let $v$ be an exterior major vertex of a graph $G$ and let $P_r$,
$P_s$ be two branches of $v$ with $l_r$ ,$l_s$ number of vertices
respectively and $l_r<l_s$. If $W$ is a WTR-set of $G$ and $W$
contains a vertex from $P_r$ other than $v$, then $W$ must contains
at least one vertex from $P_s$ other than $v$.
\end{proposition}

\begin{proof}
Suppose $u_{r,j}\in W\cap V(P_r)$ for some fix $j$; $1\le j\le l_r$
and $W\cap V(P_s)=\emptyset$. As $d(u_{s,j},w)=d(u_{r,j},w)$ for all
$w\in W\setminus \{u_{r,j}\}$, therefore the code (with respect to
$W$) of the vertex $u_{s,j}\in P_s$ lying at the $jth$ position in
$P_s$, differ by one coordinate only from the code (with respect to
$W$) of $u_{r,j}\in W$ and, as a consequence, $W$ is not a WTR-set
for $G$, by Lemma \ref{lem1}.
\end{proof}

\begin{theorem}\label{thm WTMB with Uniqu Branch}
Let $G$ be a non-path tree and $W$ be a WTMB of $G$. Let $v$ be an
exterior major vertex of $G$ with $t$ branches $\{P_i: 1\le i\le
t\}$ in which $P_1$ is the unique branch of shortest length with
$l_1$ number of vertices. Then $W$ must contains vertices
$u_{i,j}\in P_i$ for each $i$ $(2\leq i \leq t)$ and exactly one $j$
$(j>l_1)$.
\end{theorem}
\begin{proof}
Since $W$ is a resolving set for $G$ so by Remark
\ref{ThrmTerminalPathResolvingSet}, $W$ contains at least one vertex
from at least $t-1$ branches of $v$. We start by choosing $W\cap
V(P_1)=\emptyset$, i.e., $W$ does not contain any vertex from $P_1$,
then $W$ contains at least one vertex from each $P_i$ $(2\le i\le
t)$. If we take $u_{i,j}\in W\cap V(P_i)$ for some $i$ $(2\le i \le
t)$ and some $j$ where $j\le l_1$, then $d(u_{1,j},w)=d(u_{i,j},w)$
for all $w\in W\setminus \{u_{i,j}\}$, so the code of $u_{1,j}\in
P_1$ and the code of $u_{i,j}\in W$ with respect to $W$, differ by
one coordinate only, which is a contradiction that $W$ is a WTMB of
$G$. Thus for $u_{i,j}\in W\cap V(P_i)$, $j$ must be greater than
$l_1$. Also $W$ is a WTMB of $G$ if $|W\cap V(P_i)|$ is minimum
which is possible only when we take exactly one vertex from each
$P_i$ $(2\le i \le t)$. Thus if $W\cap V(P_i)=\{u_{i,j}\in P_i:$ for
each $i$ $(2\le i \le t)$ and exactly one $j$; $j>l_1\}$, then
$|W\cap V(P_i)|=t-1$ which is minimum and the codes of vertices of
$P_1$ differ by at least two coordinates from the codes of vertices
of $W$ with respect to $W$. Moreover if we choose $W \cap V(P_1)\ne
\emptyset$, then by Proposition \ref{prop2}, $|W\cap \{V(P_i): (1\le
i\le t) \}|=t$ as $l_1<l_i$ for all $(2\le i \le t)$. Thus $|W\cap
V(P_i)|$ is not minimum, which is contradiction that $W$ is a WTMB
of $G$.
\end{proof}
\begin{theorem}\label{thm WTMB without Uniqu Branch}
Let $G$ be a non-path tree and $W$ be a WTMB of $G$. Let $v$ be an
exterior major vertex of $G$ with $t$ branches $\{P_i: 1\le i\le
t\}$ and $v$ does not have the unique branch of shortest length.
Then $W$ must contains vertices $u_{i,j}\in P_i$ for each $i$
$(1\leq i \leq t)$ and exactly one $j$ $(2\le j\le l_i)$.
\end{theorem}
\begin{proof}
Since $v$ does not have the unique branch of shortest length, so
there exist $s$ $(2\le s\le t)$ similar branches $\{P_i: 1\le i\le
s\}$ of $v$, each has $l=l_i$ $(1\le i\le s)$ number of vertices.
Also $\{P_i: s< i\le t\}$ are remaining $t-s$ branches of $v$. As
$W$ is a resolving set of $G$ so by Remark
\ref{ThrmTerminalPathResolvingSet}, $W$ contains at least one vertex
from at least $t-1$ branches of $v$. We start by choosing $W\cap
V(P_1)=\emptyset$, i.e., $W$ does not contain any vertex from $P_1$,
then $W$ contains at least one vertex from $P_i$ $(2\le i\le t)$. If
we choose $u_{i,j}\in W\cap V(P_i)$ for some $i$ $(2\le i \le s)$
and some $j$ $(1< j\le l)$, then $d(u_{1,j},w)=d(u_{i,j},w)$ for all
$w\in W\setminus\{u_{i,j}\}$, thus the code of $u_{1,j}\in P_1$ and
the code of $u_{i,j}\in W$ with respect to $W$ differ by one
coordinate only, which is contradiction that $W$ is WTMB of $G$.
Thus $W\cap V(P_i)\ne\emptyset$ for each $i$ $(1\le i\le s)$.
Moreover
 $l<l_i$ for each $i$ where $(s< i \le t)$, so
 by Proposition \ref{prop2}, $W$ must contain at least one
vertex from remaining $t-s$ branches of $v$. Also $W$ is a WTMB of
$G$ if $|W\cap \{V(P_i):1\le i \le t\}|$ is minimum which is
possible only when we take exactly one vertex from each $P_i$ $(1\le
i \le t)$.
\end{proof}
\begin{corollary}\label{CorUpperboundDimWTDim} Let $G$ be a non-path
 tree and $v$ be an exterior major vertex of $G$ with $t$ branches
 and $v$ does not have the unique branch of shortest length, then
$\dim(G)\geq t-1$ and $\dim_{wt}(G)\geq t$.
\end{corollary}
\begin{corollary}\label{Cor at least one vertex from all branches}
Let $G$ be a non-path tree and $W\subseteq V(G)$ where
$W=\{u_{i,j}^k\in P_i^k$ for each $k$ $(1\le k \le ex(G))$ and each
$i$ $(1\le i\le t_k)$ and at least one $j$ $(2\le j\le l_i^k)\}$,
then $W$ is WTR-set for $G$.
\end{corollary}

It was shown by Chartrand {\em et al.} in \cite{chartrand2000} that
$\dim(G) = \sigma(G)-ex(G)$ for a non-path tree $G$. Let $\mu \ge 0$
denotes the number of exterior major vertices of $G$ which do not
have the unique branch of shortest length. The next result provide
the weak total metric dimension of a tree.
\begin{theorem}\label{thwtDimAndDimension}
Let $G$ be a non-path tree, then $\dim_{wt}(G)=\dim(G)+\mu$.
\end{theorem}
\begin{proof}
Since every WTR-set is a resolving set, so inequality \ref{eq1},
Proposition \ref{prop2} and Corollary \ref{CorUpperboundDimWTDim},
yield $\dim_{wt}(G)\ge \dim(G)+\mu$. For $\dim_{wt}(G)\le
\dim(G)+\mu$, let $\{v_k: 1\le k\le ex(G)\}$ be the set of exterior
major vertices of $G$ and each $v_k$ has $t_k$ number of branches
$P_i^k$ and each branch $P_i^k$ has $l_i^k$ number of vertices. We
label $v_k$ is ascending order of $k$ such that $v_k :(1\le k \le
\mu)$ do not have the unique branch of shortest length and $v_k
:(\mu < k \le ex(G))$ have the unique branch of shortest length
$P_1^k$. Let $B$ be a metric basis of $G$ and by Remark
\ref{ThrmTerminalPathResolvingSet}, $B$ contains at least one vertex
from at least $t_k-1$ branches of $v_k$. We choose one vertex from
branch $P_i^k$ for each $i$ $(1\le i \le t_k-1)$ and each $k$ $(1\le
k \le \mu)$ other than $v_k$ and one vertex from branch $P_i^k$ for
each $i$ $(2\le i \le t_k)$ and each $k$ $(\mu < k \le ex(G))$ other
than $v_k$. By Theorem \ref{thm WTMB with Uniqu Branch} and Theorem
\ref{thm WTMB without Uniqu Branch}, $W=B\cup \{u_{t_k,l_{t_k}^k}^k:
1\le k \le \mu\}$ is a WTMB of $G$. It concludes the proof.
\end{proof}
\begin{theorem}\label{thm1}
Let $G$ be a non-path tree. Then $\dim_{wt}(G) = 2$ if and only if
every exterior major vertex of $G$ has at most three branches in
which one is the unique branch of shortest length and one of the followings hold:\\
(1)\ $G$ has exactly one exterior major vertex with three
branches and no exterior major vertex with two branches. \\
(2)\ $G$ has exactly two exterior major vertices with two branches
and no exterior major vertex with three branches.
\end{theorem}

\begin{proof}
Given that $G$ is a non-path tree. Let $\dim_{wt}(G)=2$, since only
a path has metric dimension one \cite{khuller1996}, so inequality
(\ref{eq1}) implies that $\dim(G) = \dim_{wt}(G) = 2$ and hence $\mu
= 0$. Thus all exterior major vertices of $G$ has the unique branch
of shortest length. Suppose $G$ has an exterior major vertex with
$td(v)\ge 4$, then by Theorem \ref{thm WTMB with Uniqu Branch},
$\dim_{wt}(G)\ge 3$. Thus every exterior major vertex of $G$ has at
most three branches. We discuss the following two cases:
\\
\\
\noindent Case 1: $td(v)\le 2$ for all exterior major vertices $v$.
If $G$ has an exterior major vertex with terminal degree 2, then
there are two such vertices $v_1$, $v_2$ each has two branches
$P_1^k$, $P_2^k$ $k=1,2$ with $P_1^k$ as the unique branch of
shortest length of $v_k$ for each $k=1,2$ and by Theorem \ref{thm
WTMB with Uniqu Branch}, $\dim_{wt}(G)=2$. Suppose $G$ has another
exterior major vertex $u$ $(\ne v_1,v_2)$ with $td(u)=2$, then by
Theorem \ref{thm WTMB with Uniqu Branch}, $\dim_{wt}(G)\ge 3.$
\\
\\
\noindent Case 2: $td(v)\le 3$ for all exterior major vertices $v$.
If $G$ has an exterior major vertex $v$ with $td(v)=3$, then $v$ is
the only vertex with $td(v)=3$ and $v$ has exactly three branches
$P_1,P_2,P_3$ in which $P_1$ is the unique branch of shortest
length, and hence by Theorem \ref{thm WTMB with Uniqu Branch},
$\dim_{wt}(G)=2$. Suppose $G$ has another vertex $u$ with $td(u)=3$
or $td(u)=2$, then in both cases $\dim_{wt}(G)\ge 3$ by Theorem
\ref{thm WTMB with Uniqu Branch}.

  The converse part of the theorem is obvious.
\end{proof}
The following result provides the realization of weak total metric
dimension in some graphs $G$ of order $n$.

\begin{theorem}\label{thm6}
For every two integers $a,b$ with $2\leq a\leq b$, there exists a
graph $G$ of order $b$ with $\dim_{wt}(G) = a$.
\end{theorem}

\begin{proof}
For $a = b$, let $G$ be a graph of order $b$ in which each vertex is
twin. Then Theorem \ref{thm5}, concludes that $\dim_{wt}(G) =b =a$.
Now, consider $a\leq b-1$. Consider a path $P_{b-a+1}$ and let one
leaf of this path be $l$. Attach $a-1$ leaves $l_1,l_2,\ldots,
l_{a-1}$ with the leaf $l$. Call the resultant graph $G$ of order
$b$. For $a\geq 3$, $G$ is a non-path tree with exactly one exterior
major vertex $l$ with $td(l) = a$ in which $a-1$ are twin leaves.
Then $\mu = 1$. Also $\dim(G) = a-1$, by a result of Chartrand {\em
et al.} given in \cite{chartrand2000}. Thus, Theorem
\ref{thwtDimAndDimension} implies that $\dim_{wt}(G) = a$. For $a
=2$, $G$ is a path $P_b$ vertices and it is straightforward to see
that the set of two leaves of $P_b$ is a WTMB of $G$.
\end{proof}
\par
We define a notion which is useful for finding an upper bound on
weak total resolving number of a non-path tree.
$$ \theta(G)= \displaystyle\min_{(1\le i\le t_k),(1\le k \le ex(G))}l_i^k $$
 For instance, if a tree has twin leaves, then $\theta=2$.
\begin{proposition}\label{PropNminusThetaPlulsTwo}
 Let $G$ be a non-path tree of order $n\ge 4$ and $W\subseteq
V(G)$ with cardinality $n-\theta(G)+2$, then $W$ contains at least
one vertex from branch $P_i^k$ of an exterior major vertex $v_k$
other than $v_k$, for each $i$ $(1\le i \le t_k )$ and each $k$
$(1\le k \le ex(G))$ where $t_k$, $P_i^k$ as define earlier.
\end{proposition}
\begin{proof}
Let $G$ be a non-path tree in which $n-\theta(G)+2$ is smallest.
Also $n-\theta(G)+2$ is smallest when $n$ is smallest and
$\theta(G)$ is largest. Such a tree contains only one exterior major
vertex $v_1$ (if $G$ has more than one exterior major vertices then
$n$ is not smallest) and $v_1$ have only three similar branches
$P_1^1$, $P_2^1$, $P_3^1$ (if $G$ has more than three branches or
branches are not similar, then $n$ is not largest as compared to
$\theta(G)$) and each branch has $l^1$ is number of vertices. Then
$n=3l^1-2$ and $\theta(G)=l^1$ and $n-\theta(G)+2=2l^1$. Number of
vertices (including $v_1$) in two any two branches say $P_1^1$,
$P_2^1$ is $2l^1-1$. Thus any set $W\subseteq V(G)$ of $2l^1$
vertices contains at least one vertex from third branch $P_3^1$
other than $v_1$.
\end{proof}
\begin{theorem}\label{TheoremLowerBoundonWTResolvingNumber}
Let $G$ be a non-path tree of order $n\ge 4$, then $ res_{wt}(G)\le
n-\theta(G)+2 $.
\end{theorem}
\begin{proof}
Let $W\subseteq V(G)$ with cardinality  $n-\theta(G)+2$, then by
Proposition \ref{PropNminusThetaPlulsTwo}, $W$ contains at least one
vertex from all branches of all exterior major vertices of $G$ other
than exterior major vertices and hence by Corollary \ref{Cor at
least one vertex from all branches}, $W$ is a WTR-set. It concludes
the proof.
\end{proof}
By Proposition \ref{prop2}, Corollary \ref{CorUpperboundDimWTDim}
and Theorem \ref{thwtDimAndDimension}, we have the following
proposition for a lower bound on the weak total resolving number of
a non-path tree.

\begin{proposition}\label{prop7}
Let $G$ be a non-path tree of order $n$ with $p$ exterior major
vertices $v_1,v_2,\dots, v_p$ each has $t_k$ branches of length
$l^k_i-1$ $(1\leq k\leq p; 1\leq i\leq t_k)$. Then
$$\sum_{k=1}^p \sum_{i = 1}^{t_k} (l^k_i-1)\leq res_{wt}(G).$$
\end{proposition}

 Consider two vertices $x$ and $y$ such that $x\sim y$. Take four
paths $P_{r\geq 3}:\ x_1,x_2,\ldots,x_{r}$; $P'_{r}:\
x'_1,x'_2,\ldots,x'_{r}$; $P_3:\ y_1,y_2,y_3$ and $P'_3:\
y'_1,y'_2,y'_3$. Identify the vertex $x$ with the leaves $x_1, x'_1$
and identify the vertex $y$ with the leaves $y_1, y'_1$. The
resultant graph $G$ is a non-path tree with two exterior major
vertices $x = x_1 = x'_1$ and $y = y_1=y'_1$ each has two similar
branches. Clearly, $\mu = 2$ and so $\dim_{wt}(G) = 4$, by Theorem
\ref{thwtDimAndDimension}. Note that $p = 2$, $t_1 = 2, t_2=2$ and
so $\sum \limits_{k = 1}^{p}\sum \limits_{i = 1}^{t_k}(l^k_i-1) =
2(r+1)$. We claim that $res_{wt}(G) = 2(r+1)$. It is a routine
exercise to see that any set of cardinality $2(r+1)$ is a WTR-set
for $G$. Indeed, if we say that $res_{wt}(G)<2(r+1)$, then the set
$\left(V(P_r)-\{x_1\}\right)\cup \left(V(P'_r)-\{x'_1\}\right)\cup
\{x,y\}\cup \{v\}$, where $v\in V(P_3)-\{y_1\}$ or $v\in
V(P'_3)-\{y'_1\}$, is not a WTR-set for $G$, by Corollary
\ref{CorUpperboundDimWTDim}. It concludes that $res_{wt}(G) = \sum
\limits_{k = 1}^{p}\sum \limits_{i = 1}^{t_k}(l^k_i-1)$.

From Proposition \ref{prop7} and Theorem
\ref{TheoremLowerBoundonWTResolvingNumber} we have following
theorem.

\begin{theorem}\label{thm bounds on resolving number}
Let $G$ be non-path tree, then $$\sum_{k=1}^p \sum_{i = 1}^{t_k}
(l^k_i-1)\leq res_{wt}(G)\le n-\theta(G)+2.$$
\end{theorem}


\end{document}